\documentclass[12pt]{amsart}
\usepackage{amssymb,mathrsfs,amscd,stmaryrd,fullpage,url}

\usepackage[colorlinks]{hyperref}

\newtheorem{thm}{Theorem}[section]
\newtheorem{lem}[thm]{Lemma}

\theoremstyle{definition}

\theoremstyle{remark}

\numberwithin{equation}{section}

\def\HJ#1{\par\medskip\noindent{\bf#1.}\bgroup\it \ }
\def\EHJ{\egroup}

\newcommand{\NM}{\vartriangleleft}

\DeclareMathOperator{\Irr}{Irr}

\DeclareMathOperator{\Ker}{Ker}

\DeclareMathOperator{\GL}{GL}

\DeclareMathOperator{\cod}{cod}

\normalsize

\begin{document}

\title{Codegrees of primitive characters of solvable groups}

\author{Ping Jin, Lei Wang, Yong Yang*}

\address{School of Mathematical Sciences, Shanxi University, Taiyuan, 030006, China.}
\email{jinping@sxu.edu.cn; wanglei0115@163.com}

\address{Three Gorges Mathematical Research Center, China Three Gorges University, Yichang, Hubei 443002, China and Department of Mathematics, Texas State University, 601 University Drive, San Marcos, TX 78666, USA.}
\email{yang@txstate.edu}
\thanks{*Corresponding author}


\date{}

\begin{abstract}
We obtain the codegree of a certain primitive character for a finite solvable group,
and thereby give a negative answer to a question proposed by Moret\'o in \cite{Moreto}.
\end{abstract}

\maketitle

\section{Introduction}

Let $G$ be a finite group, and write $\Irr(G)$ to denote the set of irreducible complex characters of $G$. The concept of character codegree, originally defined as $|G|/\chi(1)$ for any nonlinear irreducible character $\chi$ of $G$,
was introduced in \cite{ch} in order to characterize the structure of finite groups.
For a nonlinear character $\chi\in \mathrm{Irr}(G/N)$ where $N$ is a nontrivial normal subgroup of $G$, however, this character will have two different codegrees when it is considered as the character of $G$ and of $G/N$, respectively.
 To eliminate this inconvenience,
 Qian, Wang, and Wei in \cite{qww} redefined the codegree
 $\mathrm{cod}(\chi)=|G:\Ker \chi|/\chi(1)$ for an arbitrary character $\chi$ of $G$.
 Up to now many properties of codegrees have been studied,
 such as variations on Huppert's $\rho$-$\sigma$ conjecture,
 the relationship between the codegrees and the element orders, groups with few codegrees, and recognizing simple groups using the codegree set, etc.

 The authors believe that among the above-mentioned results, the most interesting finding is the relation between the codegrees and the element orders (see ~\cite{I11,Q11,Qian} for example).
 Here we mention a result of Qian (\cite[Theorem 1.1]{Q11}),
 which says that if a finite solvable group $G$ has an element $g$ of square-free order,
 then $G$ must have an irreducible character of codegree divisible by the order $o(g)$ of $g$.
 In a short time, Isaacs \cite{I11} established the same result for an arbitrary finite group. Recently Qian \cite{Qian} strengthened further his previous result, showing that for every element $g$ of a finite solvable group $G$, there exists necessarily some $\chi\in {\rm Irr}(G)$ such that $o(g)$ divides ${\rm cod}(\chi)$.

Motivated by the results in ~\cite{I11,Q11,Qian}, Moret\'o considered the converse relation of codegrees and element orders and proposed an interesting question ~\cite[Question B]{Moreto}.
He also mentioned that the counterexamples, if exist, seem to be rare.

\HJ{Question}
Let $G$ be a finite solvable group and let $\chi\in\Irr(G)$.
Does there exist $g\in G$ such that $\pi(\cod(\chi))\subseteq\pi(o(g))$?
Here $\pi(n)$ denotes the set of prime divisors of a positive integer $n$.
\EHJ

In this note we will construct examples to show that this question has a negative answer in general.
To do so, we first show that any faithful primitive character of a finite solvable group $G$
has codegree divisible by all prime divisors of $|G|$ (Theorem 2.2 below),
and then present a sufficient condition to guarantee that
a solvable group possess a faithful primitive character (Theorem 2.3 below).
These two results, which we think, perhaps are of independent interest.

Throughout this paper all groups considered are finite and characters are defined over the
complex field. For notation and terminology, we refer to \cite{I76}.

\section{Codegrees of primitive characters}
We begin by reviewing some facts from $\pi$-theory of characters.
If $G$ is a $\pi$-separable group, where $\pi$ is a set of primes,
a character $\chi \in \Irr(G)$ is said to be $\pi$-special
if $\chi(1)$ is a $\pi$-number and the determinantal order $o(\theta)$ is a $\pi$-number for every irreducible constituent $\theta$ of the restriction $\chi_S$ for every subnormal subgroup $S$ of $G$. This definition was first introduced by Gajendragadkar in \cite{G1979},
and plays a prominent role in the character theory of solvable groups.
For simplicity, we say that $\chi$ is $p$-special if the set $\pi$ consists of a single prime $p$.

\begin{lem}\label{sp}
Let $G$ be a solvable group and let $\chi\in\Irr(G)$ be a nonlinear character.

{\rm (1)} If $\chi$ is primitive, then $\chi=\prod \chi_p$,
where each $\chi_p$ is $p$-special and $p$ runs over $\pi(\chi(1))$.

{\rm (2)} If $\chi$ is $p$-special for a prime $p$, then $\chi_P$ is irreducible for any Sylow $p$-subgroup $P$ of $G$.
\end{lem}
\begin{proof}
Part (1) is Theorem 2.17 of \cite{I18},
and (2) is a special case of Theorem 2.10 of the same reference.
\end{proof}

We can now establish the two results regarding the codegrees of primitive characters, as mentioned in the introduction.

\begin{thm}\label{pc}
Let $G$ be a solvable group,
and let $\chi\in\Irr(G)$ be a faithful primitive character.
Then $\cod(\chi)$ is divisible by each prime factor of $|G|$,
that is, $\pi(\cod(\chi))=\pi(|G|)$.
\end{thm}

\begin{proof}
If $\chi(1)=1$, then $\cod(\chi)=|G|$ by definition, and there is noting to prove.
So we may assume that $\chi(1)>1$.
By Lemma \ref{sp}(1), we can write $\chi=\prod \chi_p$, where $p$ runs over $\pi(\chi(1))$
and $\chi_p$ is $p$-special.
Note that $\cod(\chi)=|G|/\chi(1)$,
and to complete the proof, therefore, it suffices to show that $\cod(\chi)$ can be divisible by every prime in $\pi(\chi(1))$.

Suppose that $p$ does not divide $\cod(\chi)$ for some $p\in\pi(\chi(1))$.
Then $|G|/\chi_p(1)$ is not divisible by $p$, so $\chi_p$ has $p$-defect zero.
In this case, it is well known that $\chi_p$ vanishes all nontrivial elements of a Sylow $p$-subgroup $P$ of $G$ (see Theorem 8.17 of \cite{I76} for example),
but then, by Lemma \ref{sp}(2), we see that the restriction $\chi_P$ is irreducible.
This cannot be the case, thus proving the theorem.
\end{proof}

\begin{thm}\label{cod}
Let $G$ be a solvable but non-nilpotent group, and assume that the following conditions hold.

{\rm (1)} $Z(G)$ is a cyclic group.

{\rm (2)} $F(G)/Z(G)$ is a chief factor of $G$.

{\rm (3)} $|F(G)|$ and $|G/F(G)|$ are relatively prime.\\
Then there is a faithful primitive character $\chi\in\Irr(G)$  such that $$\cod(\chi)=|G|/|F(G):Z(G)|^{\frac{1}{2}}.$$
If in addition $G/F(G)$ has no element of order divisible by all prime divisors of $|G/F(G)|$,
then $\pi(\cod(\chi))$ is not contained in $\pi(o(x))$ for every $x\in G$.
\end{thm}

\begin{proof}
We will complete the proof by carrying out the following steps.

\smallskip
\emph{Step 1. We fix a faithful linear character $\lambda$ of $Z(G)$ and let $\theta\in\Irr(F(G))$ lie over $\lambda$. Then $\theta$ is faithful and fully ramified with respect to $F(G)/Z(G)$,
so that $\theta(1)^2=|F(G):Z(G)|$.}

The existence of such faithful linear character $\lambda$ of $Z(G)$ is guaranteed by Condition (1).
Note that $Z(G)\le Z(F(G))$ and $Z(F(G))/Z(G)\NM G/Z(G)$, so $Z(G)=Z(F(G))$ by Condition (2).
For any $\theta\in\Irr(F(G))$ lying over $\lambda$,
since $\theta_{Z(G)}=\theta(1)\lambda$, we see that $\Ker\theta\cap Z(G)=\Ker\lambda=1$.
But now $Z(G)$ is the center of the nilpotent group $F(G)$,
which forces $\Ker\theta=1$, so $\theta$ is faithful.
By Theorem 2.31 of \cite{I76}, it follows that $\theta$ is fully ramified with respect to $F(G)/Z(G)$,
and thus $\theta(1)^2=|F(G):Z(G)|$.

\smallskip
\emph{Step 2. The character $\theta$ is extendible to $G$, and every extension $\chi\in\Irr(G)$ of $\theta$ is primitive.}

By Step 1, we know that $\theta$ is the unique irreducible character of $F(G)$ that lies over $\lambda$,
so $\theta$ is necessarily $G$-invariant, and since $F(G)$ is a normal Hall subgroup of $G$ by Condition (3), it follows that $\theta$ is extendible to $G$ by Gallagher's extension theorem (see Theorem 8.15 of \cite{I76}). Let $\chi\in\Irr(G)$ be an extension of $\theta$ and suppose that $\chi=\alpha^G$, where $\alpha\in\Irr(U)$ for some subgroup $U$ of $G$.
Write $D=F(G)\cap U$. Then $Z(G)\le U$ (this is an elementary fact about characters; see Problem 5.12 of \cite{I76} or Lemma 10.38 of \cite{I18} for a proof), so $Z(G)\le D$.
Since $(\alpha^G)_{F(G)}=\chi_{F(G)}=\theta$, which is irreducible, we conclude that $G=F(G)U$ and hence $D\NM G$. Now by Condition (2), we have two possibilities: either $D=Z(G)$ or $D=F(G)$.
Assume that $D=Z(G)$. Then $(\alpha_{Z(G)})^{F(G)}=(\alpha^G)_{F(G)}=\theta$,
which forces $\alpha_{Z(G)}=\lambda$ and hence $\lambda^{F(G)}=\theta$. This is impossible by Step 1,
so $D=F(G)$. Now we have $U=G$, and thus $\chi$ is primitive.

\smallskip
\emph{Step 3. Each extension $\chi\in\Irr(G)$ of $\theta$ is faithful,
and thus $\cod(\chi)=|G|/|F(G):Z(G)|^{\frac{1}{2}}$.}

Let $K=\Ker\chi$. Since $\theta$ is faithful by Step 1,
we see that $K\cap F(G)=\Ker\theta=1$, and thus $K$ centralizes $F(G)$ and hence $F(G)/Z(G)$.
But $F(G)/Z(G)=F(G/Z(G))$, which is self-centralized in the solvable group $G/Z(G)$,
so $K\le F(G)$ and $K=1$. This proves that $\chi$ is faithful,
and by definition, we have
$$\cod(\chi)=|G:\Ker\chi|/\chi(1)=|G|/\theta(1)=|G|/|F(G):Z(G)|^{\frac{1}{2}}.$$

\smallskip
\emph{Step 4. Complete the proof.}

We have established that $G$ has a faithful primitive character $\chi$, so Theorem \ref{pc} applies,
and we obtain $\pi(\cod(\chi))=\pi(|G|)$.
Assume that $\pi(\cod(\chi))\subseteq\pi(o(g))$ for some element $g\in G$.
Then $\pi(o(g))=\pi(|G|)$.
Consider the quotient group $\bar G=G/F(G)$. It is clear that the image $\bar g$ of $g$ in $\bar G$
also has the same property that $\pi(o(\bar g))=\pi(|\bar G|)$.
This shows that the last statement of the theorem also holds, and the proof is now complete.
\end{proof}

As an application of Theorem ~\ref{cod}, we can now construct a counterexample to Moret\'o's question.
For the definition and properties of semi-linear groups $\Gamma(q^n)$, we refer to \cite{manz/wolf}.

Let $T$ be a subgroup of the semi-linear group $\Gamma(2^{10})$,
where $|\Gamma(2^{10})|=(2^{10}-1)\cdot 10$.
The normal cyclic subgroup of $T$ from the bottom $2^{10}-1$ part is of order $2^5+1=33=3 \cdot 11$,
and the Galois group of $T$ from the top $10$ part is of order $5$, so that $|T|=3 \cdot 5 \cdot 11$.
Since $T$ is not cyclic by a direct calculation,
it follows that $T$ has no element of order divisible by all three primes of $|T|$.

Our group $G$ is the semi-direct product $E\rtimes T$, where $E$ is an extra-special group of order $2^{11}$
such that $T$ acts faithfully on $E$ and irreducibly on $E/Z(E)$.
We point out that the group $G$ can be realized as a subgroup of $\GL(2^5,3)$, and it has been constructed explicitly in GAP ~\cite{GAP2020}.
It is easy to see that $F(G)=E$ and $G/F(G)\cong T$, so by Theorem ~\ref{cod},
we conclude that Moret\'o's question can have a negative answer.

Using GAP we can calculate the codegree set of the group $G$, which is \[\{1,3,5,11,15,33,1024,2112,5120,10560 \}.\]

Of course one can construct many other examples of the same flavor.

\bigskip

\noindent
\textbf{Acknowledgement} This work was supported by NSFC (Grant No. 12171289) and a grant from the Simons Foundation (No. 499532). 


\end{document}